\documentclass[a4paper, 12pt]{amsart}

\usepackage{vmargin}
\usepackage[colorlinks=true,linkcolor=blue,citecolor=blue,urlcolor=blue]{hyperref}
\usepackage{bookmark}
\usepackage{amsthm,thmtools,amssymb,amsmath,amscd,amsfonts}
\usepackage{mathrsfs}
\usepackage{stmaryrd}

\usepackage{fancyhdr}
\usepackage{esint}

\usepackage{enumerate}

\usepackage{pictexwd,dcpic}

\usepackage{graphicx}
\usepackage[utf8]{inputenc}

\declaretheorem[name=Theorem,numberwithin=section]{thm}
\declaretheorem[name=Remark,style=remark,sibling=thm]{rem}
\declaretheorem[name=Lemma,sibling=thm]{lemma}
\declaretheorem[name=Proposition,sibling=thm]{prop}
\declaretheorem[name=Definition,style=definition,sibling=thm]{defn}

\numberwithin{equation}{section}

\usepackage{cleveref}
\crefname{lemma}{Lemma}{Lemmata}
\crefname{prop}{Proposition}{Propositions}
\crefname{thm}{Theorem}{Theorems}
\crefname{cor}{Corollary}{Corollaries}
\crefname{defn}{Definition}{Definitions}
\crefname{example}{Example}{Examples}
\crefname{rem}{Remark}{Remarks}
\crefname{ass}{Assumption}{Assumptions}
\crefname{not}{Notation}{Notation}
\crefname{section}{Section}{Sections}

\newcommand{\R}{\mathbb{R}}

\newcommand{\cW}{\mathcal{W}}

\DeclareMathOperator{\Tr}{Tr}
\DeclareMathOperator{\Id}{Id}

\DeclareMathOperator{\osc}{osc}

\DeclareMathOperator{\grad}{grad}

\protected\def\ignorethis#1\endignorethis{}
\let\endignorethis\relax

\DeclareMathOperator{\met}{Met}

\newcommand{\mdot}{\cdotp}

\newcommand{\cbrac}[1]{\left(#1\right)}

\newcommand{\dbrac}[1]{\left\{#1\right\}}
\newcommand{\modulus}[1]{|#1|}

\newcommand{\set}[1]{\dbrac{#1}}

\newcommand{\dom}{ {\mathcal{D}}}
\newcommand{\ran}{ {\mathcal{R}}}

\newcommand{\comp}{\circ}

\newcommand{\Na}{\ensuremath{\mathbb{N}}}

\newcommand{\script}[1]{\mathscr{#1}}

\renewcommand{\emptyset}{\varnothing}
\newcommand{\union}{\cup}
\newcommand{\Union}{\bigcup}
\newcommand{\intersect}{\cap}

\newcommand{\close}[1]{\overline{#1}}		

\renewcommand{\epsilon}{\varepsilon}

\renewcommand{\phi}{\varphi}

\newcommand{\embed}{\hookrightarrow}		

\newcommand{\norm}[1]{\| #1 \|}			
\newcommand{\spt}[1]{{\rm spt} {\text{ }}#1}	
\DeclareMathOperator{\esssup}{esssup}
\DeclareMathOperator{\essinf}{essinf}

\newcommand{\interior}[1]{\mathring{#1}}	

\DeclareMathOperator{\tr}{tr}			

\DeclareMathOperator{\divv}{div}		

\newcommand{\pullb}[1]{{#1}^\ast}			

\DeclareFontFamily{OT1}{restrictfont}{}
\DeclareFontShape{OT1}{restrictfont}{m}{n}{<-> fmvr8x}{}

\newcommand{\adj}[1]{{#1}^\ast}			

\newcommand{\inprod}[1]{\langle #1 \rangle}	
\newcommand{\Leb}[1][{}]{\script{L}^{#1}}			

\newcommand{\conj}[1]{\overline{#1}}				

\newcommand{\Lp}[2][{}]{{\rm L}^{#2}_{\rm #1}}		
\newcommand{\Ck}[2][{}]{{\rm C}^{#2}_{\rm #1}}		
\newcommand{\Sob}[2][{}]{{\rm W}^{#2}_{\rm #1}}		
\newcommand{\SobH}[2][{}]{{\Sob[#1]{#2,2}}}	

\newcommand{\sE}{\script{E}}

\newcommand{\hk}{\rho}

\DeclareMathOperator{\sgn}{sgn}

\begin{document}

\title[Heat kernels and regularity for rough metrics]{Heat kernels and regularity for rough metrics on smooth manifolds}

\author{Lashi Bandara}
\author{Paul Bryan}

\address{Lashi Bandara, 
Institut für Mathematik,
Universität Potsdam, 
D-14476, Potsdam OT Golm, Germany
}
\urladdr{\href{http://www.math.uni-potsdam.de/~bandara}{http://www.math.uni-potsdam.de/~bandara}}
\email{\href{mailto:lashi.bandara@uni-potsdam.de}{lashi.bandara@uni-potsdam.de}}

\address{Paul Bryan, 
Department of Mathematics
Macquarie University 
NSW, 2109, Australia
}
\urladdr{\href{http://pabryan.github.io}{http://pabryan.github.io/}}
\email{\href{mailto:paul.bryan@mq.edu.au}{paul.bryan@mq.edu.au}}

\curraddr{}
\email{}
\date{\today}

\dedicatory{}

\maketitle

\begin{abstract}
We consider rough metrics on smooth manifolds and corresponding Laplacians induced by such metrics.
We demonstrate  that globally continuous heat kernels exist and are Hölder continuous locally in space and time. 
This is done via local parabolic Harnack estimates for weak
solutions of operators in divergence form with bounded measurable coefficients 
in weighted Sobolev spaces.
\end{abstract}

\setcounter{tocdepth}{1}
\tableofcontents

\parindent0cm
\setlength{\parskip}{\baselineskip}

\section{Introduction}
\label{sec:intro}
The existence and regularity of heat kernels on smooth manifolds with smooth metrics, compact or noncompact, 
is now a matter of classical fact. 
However, it is also useful and interesting to consider this problem on
metrics with non-smooth, and even discontinuous coefficients. 
Such metrics can arise naturally and the prototypical example 
is when the metric tensor is obtained as a pullback of a smooth metric under a 
Lipeomorphism. A metric of this form will in general possess measurable
coefficients.

Indeed, there has been some progress in this direction,
with two notable works being \cite{Norris, ERS}.
However, the focus of their work is somewhat different to 
what we present here, as are their methods.
We consider a wide and useful class of 
low regularity metrics called \emph{rough metrics}, 
which by definition, have only measurable coefficients, and are Riemannian-like in the sense
that they are locally comparable almost-everywhere to a smooth Riemannian metric (see \Cref{defn:rough_metric}).
These metrics became of interest in \cite{BMc, BRough} as they constitute geometric invariances of the Kato square
root problem. They are particularly significant in the noncompact setting.  
In the compact case, these metrics were used to study regularity properties
of a geometric flow, weakly tangential to the Ricci flow, 
in \cite{BLM, BCont}.

Unlike their classical, smooth counterparts, we are not able
to construct distances for rough metrics via the length functional as 
this is not a well defined device in this setting.
Nevertheless, rough metrics induce a Borel measure $\omega_g$ and we can define energies by the use 
of the exterior derivative. Consequently,  they give rise to natural Laplacians as self-adjoint operators.

The underlying manifolds we consider are smooth (topologically), 
the exterior derivative of \(C^{\infty}\) functions (and  differential forms)
 is defined, and is closable 
in $\Lp{2}(M)$, allowing us to construct Sobolev spaces $\SobH{1}(M)$
and $\SobH[0]{1}(M)$. The
former space is the energy space for the Neumann Laplacian
and the latter for the Dirichlet counterpart.
More generally, we can consider a closed subspace $\cW$ of $\SobH{1}(M)$
satisfying $\SobH[0]{1}(M) \subset \cW \subset \SobH{1}(M)$
as an energy space for the so-called \emph{mixed boundary conditions}. 
Corresponding to each such space, we obtain a Laplacian
as a self-adjoint operator and we can now ask
whether a heat kernel exists, and how regular one
can expect such an object to be. In the setting of rough metrics, our main theorem
is the following. 

\begin{thm}
\label{Thm:Main}
On a smooth manifold \(M\), equipped with a rough metric \(g\), 
and a subspace $\SobH[0]{1}(M) \subset \cW \subset \SobH{1}(M)$
with $\Ck{\infty} \intersect \cW$ dense in $\cW$, 
there exists a heat kernel \(\hk_t^{g,\cW}\) satisfying:  
\begin{enumerate}[(i)] 
\item \(\hk_t^{g,\cW} > 0\) for \(t > 0\),
\item on a parabolic cylinder \(Q = K \times [t_1, t_2]\) with \(0 < t_1 < t_2\), \(K \subseteq M\) is compact, there exists an \(\alpha = \alpha(Q)\) so that \(\hk_t^{g,\cW} \in C^{\alpha}(K \times K)\) for \(t_1 \leq t \leq t_2\).
\end{enumerate}
\end{thm}

Note that heat kernels have been considered
for some time in more general settings than smooth Riemannian manifolds,
for instance on metric spaces with bounds
on certain synthetic notions of curvature (c.f. \cite{Sturm}).
Many classical results can be recovered from this general 
theory for the special case of manifolds, 
since smooth (or continuous) metrics induce an associated intrinsic
distance structure. 

The treatment of heat kernels on smooth manifolds with smooth metrics 
usually proceed by constructing
a so-called \emph{minimal} heat kernel via local-to-global methods.
However, additional assumptions, typically on curvature,
are required in order to obtain its uniqueness  (see \cite{Chavel}).
More seriously, the following Varadhan's formula
$$ d^2(x,y) = \lim_{t \to 0} 4 t \log \hk_t(x,y)$$
may fail for this minimal heat kernel. 
In this classical setting, (or even for certain classes of 
metric spaces), it is well known that Varadhan's formula holds for the heat 
kernels corresponding to the Dirichlet and Neumann Laplacians 
(see  \cite{Norris, ERS}).

As aforementioned, the results of \cite{Norris, ERS} 
are similar to ours although their methods are different. 
Our approach is somewhat simpler with the key 
idea being to construct a global heat kernel
via the Riesz representation theorem, a perspective made known to us from 
\cite{Davies} in the smooth setting. However, given the generality of our setting,
 we require a certain 
weak Harnack type estimate, along with 
some operator theoretic facts, for this technique to succeed.

This approach to constructing a heat kernel was used previously in  
\cite{BCont} where the manifold was assumed to be compact,  and 
the existence and regularity of the heat kernel was
reduced to \emph{parabolic Harnack estimates} for divergence form operators
with bounded measurable coefficients against a smooth background. 
The required estimates
were obtained by observing that  
the rough metric is globally comparable to a smooth one,
due to the compactness of the underlying manifold, and
through the  results  in \cite{SC}.
In our situation,  we cannot argue in this way without imposing
severe restrictions. Thus, we demonstrate
how to obtain the heat kernel via \emph{local Harnack estimates} 
before proceeding to show that such estimates 
hold in our setting.

\section*{Acknowledgements}

The first author was supported by the Knut and Alice Wallenberg foundation, KAW 2013.0322 
postdoctoral program in Mathematics for researchers from outside Sweden, and from  SPP2026 
from the German Research Foundation (DFG).
The second author was 
supported by the  EPSRC on a Programme Grant entitled ``Singularities of Geometric Partial Differential Equations'' 
reference number EP/K00865X/1 as well as the University of Queensland and Joe Grotowski.
Both authors would like to thank Warwick University where this project began as
well as Kaj Nyström for his useful comments.
Moreover, they would like to acknowledge the gracious support of the first author's 
parents, Jayanthi and Mano Bandara, for providing a week of curry and hospitality during some of the crucial 
phases of this paper, along with the Orbost Hotel in Victoria for their beer and ambience. 

\section{Rough Metrics}
\label{sec:rough_metrics}

As far as the authors are aware, the term ``rough metric'' as used in the current context 
was coined in \cite{BRough} as they were 
recognised to be the geometric invariances of the Kato square root problem.
We emphasise here that similar notions existed implicitly in the 
literature  in \cite{Norris,SC}.
We recall the most important aspects of rough metrics here. 
A more detailed exposition can be found in Section 3 in \cite{BRough}.

We begin with the following definition, which recognises
that a manifold affords us not only with a metric independent 
topology and differentiable structure, but also a measure
structure that is independent of any metric.

\begin{defn}
\label{defn:borel_lebesgue}
Let \(M^n\) be a smooth, \(n\)-dimensional manifold. A set \(B \subseteq M\) is a \emph{Borel set} if for all charts \(\varphi : U \subseteq M \to V \subseteq \R^n\), the set \(\varphi(B \cap U) \subseteq \R^n\) is a Borel set. A set \(E \subseteq M\) is said to be \emph{Lebesgue measurable} if \(\varphi(E \cap U)\) is a Lebesgue measurable set of \(\R^n\).
\end{defn}

\begin{rem}
The Lebesgue measurable sets form a \(\sigma\)-algebra containing the Borel \(\sigma\)-algebra. The notion of Lebesgue measurable set does not involve a measure on \(M\). The collection of Lebesgue measurable sets is canonical in the following sense: given any smooth (or continuous) metric $h$, this collection of subsets are exactly the collection of $\omega_h$-measurable sets.
\end{rem}

We now define the notion of measurable functions and sections of smooth vector bundles.

\begin{defn}[Measurable functions and sections]
A function $f: M \to \R$ is measurable if $f^{-1}(-\infty, \alpha]$ is measurable for every $\alpha  \in \R$. A measurable section \(s\) of a vector bundle \(\pi : V \to M\) is a function \(s: M \to V\) with \(\pi \circ s = \Id_M\) and such that for any Lebesgue measurable subset \(E \subseteq V\), the set \(s^{-1}(E)\) is a Lebesgue measurable subset of \(M\).  We denote the set of such sections by $\Gamma(V)$.
\end{defn}

Note that \(V\) is a smooth manifold itself and so Lebesgue Measurability of subsets of \(V\) is defined as in \ref{defn:borel_lebesgue}. In particular, we have the bundle $\pi_{p,q} : T^{p,q}M \to M$ of $(p,q)$ tensors of covariant rank $p$ and contravariant rank $q$ as well as measurable tensor fields \(\Gamma(T^{p,q}M)\). We also have the bundle of differential forms \(\pi_k : \Lambda^k T^{\ast} M \to M\) with measurable sections \(\Omega^k(M) = \Gamma(\Lambda^k T^{\ast} M)\). The exterior derivative is however, \emph{not defined} for all measurable sections. All the usual constructions of smooth sub-vector and sub-fibre bundles apply and in particular, we have a well defined notion of measurable sections of the fibre bundle, \((T^{\ast} M \odot T^{\ast} M)_+\) of positive definite, symmetric bilinear forms important in the main definition of this section.

\begin{defn}[Rough metric]
\label{defn:rough_metric}
A \emph{rough} metric \(g\) is a Lebesgue measurable section of \(\met(M) = (T^{\ast} M \odot T^{\ast} M)_+\), the bundle of positive definite, symmetric bilinear forms on \(TM\) that are in addition, \emph{locally comparable} to Euclidean metrics:
for each $x \in M$, there is a chart $\psi_x:V_x \to \R^n$ and a constant $C_x = C_x(V_x)$ such that
\[
\frac{1}{C_{x}} \pullb{\psi_x}\delta_{\R^n}(y) (X, Y) \leq g(y)(X, Y) \leq C_{x} \pullb{\psi_x}\delta_{\R^n}(y) (X, Y)
\]
for all tangent vectors \(X, Y \in T_yM\) for almost all $y \in V_x$. 
\end{defn}

\begin{rem}
Equivalently, there is an open cover \(\lbrace U_{\alpha} \rbrace\) of \(M\), smooth metrics \(g_{\alpha} \in \met(U_{\alpha})\) and real constants \(C_{\alpha} > 0\) such that
\[
\frac{1}{C_{\alpha}} g_{\alpha} (X, Y) \leq g(X, Y) \leq C_{\alpha} g_{\alpha} (X, Y)
\]
for all tangent vectors \(X, Y \in T_yM\) for almost every $y \in U_\alpha$.
\end{rem}

\begin{rem}
By employing a partition of unity argument, we can patch together the metrics \(g_{\alpha}\) to produce a globally defined, smooth metric \(h\). If the constants \(C_{\alpha}\) are uniformly bounded above and away from zero, then our rough metric will be globally comparable to \(h\). This is automatic whenever \(M\) is compact. If \(M\) is not compact, then this need not be true.
However, if the rough metric $g$ is at least continuous, we can find a smooth globally comparable metric for any $C > 1$. 
\end{rem}

Given a rough metric, we may locally define a \(g_{\alpha}\)-self-adjoint bounded measurable section \(A_{\alpha} : U_{\alpha} \to T^{1,1}U_{\alpha}\) by
\[
g(X, Y) = g_{\alpha} (A_{\alpha} (X), Y).
\]
Such a definition has the advantage that it allows us to work with respect to the \emph{smooth} metric \(g_{\alpha}\), 
at least locally.

A rough metric gives rise to an $\Lp{p}$-theory over the $T^{p,q}M$ tensor bundle by defining \(L^p\) norms,
\[
\|\xi\|_p := \left(\int_{M} |\xi|_{g}^p\ d\omega_g\right)^{1/p}
\]
for $p \in [1, \infty)$ and
\[
\|\xi\|_{\infty} := \inf\set{C: |\xi|_g \leq C\ \text{a.e.}}
\]
where the rough metric \(g\) is extended to tensor bundles exactly as in the smooth case. For example, \(g(X \otimes Y, Z \otimes W) = g(X, Y) g(Z, W)\) gives a well defined, measurable section of positive definite, symmetric bilinear forms on \(TM \otimes TM\).

Lastly, let us note that all function spaces we consider are complex-valued function spaces, which are obtained from the real setting by complexification.

\section{Laplacians and Heat Equation for rough metrics}

For a smooth function \(f \in \Ck{\infty}(M)\), write \(\nabla f = df\) for the differential of \(f\). 
The differential \(\nabla : \Ck{\infty}(M) \to \Ck{\infty}(T^{\ast} M)\), maps smooth functions 
\(\Ck{\infty}(M)\) to smooth one-forms \(\Ck{\infty}(T^{\ast} M)\). 
We emphasise that this object is only dependent on the differentiable structure of $M$, and is independent of any choice of metric.

Recall that in the classical setting of a smooth $g$,  the Laplacian acting on functions is given  via the expression $\Delta_g = - \tr_g \nabla^2$, where $\nabla^2 = \nabla^{T^\ast M} \comp \nabla$
with $\nabla^{T^\ast M}$ the Levi-Civita connection of $g$ on \(T^{\ast} M\) and the trace is taken on the \((1,1)\) hessian obtained by metric contraction of the \((2, 0)\) hessian \(\nabla^2\). Equivalently, we obtain that  that \(\Delta_g = \grad_g^\ast \grad_g\) where \(\grad_g\) is the metric contraction of \(\nabla\).

For rough metrics, or in fact, any metric below $\Ck{0,1}$, we have no notion of metric-compatibility (being unable to differentiate such a metric), and hence the Levi-Civita connection is not generally defined. Therefore, we understand this operator in an appropriate weak sense as follows. 

First, we observe that for a rough metric $g$, the operator
$\nabla_p := d_p: \Ck{\infty} \cap \Lp{p}(M) \to \Ck{\infty} \cap \Lp{p}(T^\ast M)$
is closable, as well as $\nabla_c := d_c$ with $\dom(d_c) = \Ck[c]{\infty}(M)$,
for $p \in [1, \infty)$. A proof of this statement is given in Proposition 3.10 in \cite{BRough}, 
which reduces to covering the manifold via precompact locally comparable
charts and noting that $d$ commutes with pullbacks inside each such chart.
Consequently, we can define first-order Sobolev spaces
\[
\Sob[0]{1,p}(M) := \dom(\overline{\nabla_c})\quad \text{and}\quad  \Sob{1,p}(M) := \dom(\overline{\nabla_p}),
\]
where the closures are taken with respect to the \(\Lp{p}(M)\) norm. We may then define Sobolev norms
\[
\|u\|_{\Sob{p}} = \|u\|_p + \|\nabla u\|_p
\]
which are finite on the respective function spaces $\Sob[0]{1,p}(M) \subset \Sob{1,p}(M) \subset \Lp{p}(M)$. We remark that without closability, if these spaces are obtained via a completion with respect to the Sobolev norm, it is unclear that they are, in fact, spaces of functions. In this paper, we exclusively deal with \(p = 2\) and another consequence of the closability, coupled with the 
fact that $\Ck[c]{\infty}(M)$ is dense in $\Lp{2}(M)$,
is that $\nabla_2$ and $\nabla_c$ are densely-defined
operators. Therefore, operator theory yields 
that $\adj{\nabla_2}$ and $\adj{\nabla_c}$ exist
as densely-defined, closed operators.

\subsection{The Laplacian in the smooth setting}

In the case of a smooth $g$ that is also complete, we
always have that $\SobH[0]{1}(M) = \SobH{1}(M)$. 
Since $\nabla_c = \nabla_2$, we obtain 
a unique Laplacian $\Delta_g = \adj{\nabla_c} \close{\nabla_c} = \adj{\nabla_2}\close{\nabla_2}$.
We will emphasise at this moment that even in the smooth case, this does not mean that the Laplacian is essentially self-adjoint on $\Ck[c]{\infty}(M)$; in fact, the only general statement that can be made is that
$\SobH{2}(M) \subset \dom(\Delta_g)$. The 
case of essential self-adjointness can
be obtained under a uniform lower bound on Ricci
curvature (see \cite{BDensity}) but it
is not known to the authors whether this result is sharp.

It is also useful to obtain $\Delta_g$ via the energy,
$$\sE_g(u,v) = \inprod{\nabla u, \nabla v} = \int_M\ g(\nabla u, \nabla v)\ d\omega_g,$$
where $\dom(\sE_g) = \SobH[0]{1}(M) = \SobH{1}(M)$.
The operator $\Delta_g$ is now obtained via the so-called first and second representation 
theorems, Theorems 2.1 and 2.23 in Chapter IV  in \cite{Kato}.
 
Note that in our setup, $\nabla$ is independent of the geometry (i.e. the metric), but the energy is not. At the operator level, it is in taking the adjoint $\adj{\nabla}_g$ where the geometry becomes of consequence. Alternatively, defining the gradient, $\grad_g u =  (\nabla u)^\sharp = g(\nabla u, \cdot)$, we have by definition, $\sE_g(u,v) = \inprod{\grad_g u, \grad_g v}$. In this case, $\Delta_g = \adj{\grad}_g \close{\grad}_g$, and so we could equivalently define \(\Delta_g\) via \(\grad_g\). We prefer the former picture for the simple fact that in the latter picture, the metric information and topological information are intertwined whereas in our case we have two operators in which one which only depends on the differential structure (the exterior derivative $\nabla$) and the other on the geometry $g$ (the adjoint of the exterior derivative $\nabla^\ast_g$). In other words, there is a canonical
operator, $\nabla$ arising from the smooth structure and \emph{each geometry (i.e. metric) determines an adjoint $\nabla^{\ast}_g$}.

When $g$ fails to be complete, it may be that  
$\SobH[0]{1}(M) \subsetneqq  \SobH{1}(M)$.
In that case, we obtain a \emph{Dirichlet Laplacian} and  a \emph{Neumann 
Laplacian} corresponding to which space we pick 
to consider the associated energy.
We retain this language from the world of boundary 
value problems because there, when $M = \interior\Omega$
for a bounded domain $\Omega$ (say with Lipschitz boundary), we have that
$\SobH[0]{1}(M)$ defines the energy for
the Laplacian considered in the Dirichlet problem
and $\SobH{1}(M)$ defines the energy 
for the Laplacian considered for the Neumann problem.

\subsection{The Laplacian in the rough setting}
Inspired by the classical setting, we define
the following.
 
\begin{defn}[$(g,\cW)$-Laplacian]
Let \(g\) be a rough metric on a smooth manifold \(M\)
and $\cW \subset \SobH{1}(M)$ be a closed subspace
of $\SobH{1}(M)$ such that $\Ck[c]{\infty}(M) \subset \cW$
and $\Ck{\infty} \cap \cW$ is dense in $\cW$ with respect to the norm $\norm{\cdot}_{\cW} = \norm{\cdot}_{\SobH{1}}$. The \emph{$(g,\cW)$-Laplacian}
\(\Delta_{g,\cW} : \Lp{2}(M) \to \Lp{2}(M)\) is defined by
\[
\Delta_{g,\cW} u := \adj{\nabla_{\cW}} \close{ \nabla_{\cW}} u,
\]
where  $\nabla_{\cW} = \nabla$ with $\dom(\nabla_{\cW}) = \cW$. 
The domain of the operator $\Delta_{g,\cW}$ is then given by
$\dom(\Delta_{g,\cW}) = \set{u \in \cW: \modulus{\inprod{\nabla u, \nabla v}} \lesssim C_u \norm{v}\ \forall v \in \cW}.$
\end{defn}

By construction, the  operators $\Delta_{g,\cW}$ 
are densely-defined, closed, self-adjoint and satisfy 
$\dom(\sqrt{\Delta_{g,\cW}}) = \cW$.
It is difficult to see how to obtain this from the expression for $\Delta_{g,\cW}$.
Rather, this operator is constructed via the first and second representation 
theorems, Theorems 2.1 and 2.23 in Chapter IV in \cite{Kato}, on considering the   
energy $\sE_{g,\cW}(u,v) = \inprod{\nabla u, \nabla v}$ 
with $\dom(\sE_{g,\cW}) = \cW$ written exactly as we did for the smooth case.
These are routine facts from operator theory that are valid in far greater generality 
than what we consider here and these methods are exposited to 
greater depths in \cite{Yosida, Kato}.
 
Let us remark on why allow for the spaces \(\cW\). First, observe that the case of $\cW = \SobH[0]{1}(M)$
yields the Dirichlet Laplacian
and $\cW = \SobH{1}(M)$ yields the Neumann counterpart.
However, beyond these two obvious choices, 
there are many interesting spaces $\cW$ that
can be considered. These are best
seen emerging from boundary value problems.
As a guiding example, let $M = \Omega \subset \R^n$
be a smooth (or Lipschitz) bounded domain, and let $\Sigma \subset \partial \Omega$
be an open subset of the boundary $\partial \Omega$.
On letting $\Tr: \SobH{1}(\Omega) \to \SobH{\frac{1}{2}}(\partial \Omega)$
be the trace map to the boundary, define 
$\cW_{\Sigma} = \set{u \in \SobH{1}(\Omega): \spt (\Tr u) \subset \close{\Sigma}}$. 
It is clear that $\Ck[c]{\infty}(\Omega) \subset \cW_{\Sigma}$ and
it can be shown that $\cW_{\Sigma} \subset \SobH{1}(\Omega)$
is a closed subset. 
The interpretation here is that $\Sigma$ specifies Neumann boundary conditions whereas $\Omega \setminus \Sigma$
specifies Dirichlet boundary conditions. 
These are the so-called mixed boundary value problems 
introduced to us through \cite{AKM2}.

\subsection{The Heat Equation}

Now, we proceed to define what we mean by a solution to the heat
equation, which becomes the central theme of what is to follow.

\begin{defn}[Solution to the heat equation]
A function \(u \in \Ck{1}((0,\infty),  \dom(\Delta_{g,\cW}))\) solves the \emph{\(\cW\)-heat equation} (or just heat equation for short) with initial condition $u_0 \in \Lp{2}(M)$ if we have that
\begin{enumerate}[(i)]
\item  $\partial_t u (\cdot, t) = \Delta_{g,\cW} u(\cdot, t)$ for $t \in (0, \infty)$ and 
\item $\lim_{t \to 0} u(\cdot, t) = u_0$ in $\Lp{2}(M)$.
\end{enumerate}
\end{defn}

For any solution \(u\), of the heat equation, we have the representation formula
\[
u(x,t) = (e^{-t \Delta_{g,\cW}}u_0)(x)
\]
for almost every $x \in M$.
 
\begin{rem}
In the compact case with a smooth metric, ``a solution'' is often formulated to mean $u \in \Ck{\infty}( M \times (0, \infty))$
with $\partial_t u(x,t) = -\tr \nabla^2 u(x,t)$ with 
$\lim_{t \to 0} u(x,t) = u_0 \in \Ck{\infty}(M)$.
Since $\Ck{\infty}(M) \subset \Lp{2}(M)$ this notion of solution is stronger than our notion of solution. Here, a solution is a $\Ck{1}$ map from $(0, \infty)$ to $\dom(\Delta_{g,\cW}) \subset \cW \subset \SobH{1}(M)$ with the initial condition attained in the $\Lp{2}(M)$ topology.
\end{rem}

\section{Existence and positivity of the heat kernel}

\begin{defn}[Heat kernel]
A separably measurable map $(t,x,y) \mapsto \hk^{g,\cW}_{t}(x,y): \R_+ \times M \times M \to \R$, 
almost-everywhere symmetric in $(x,y)$, is a heat kernel 
if for every $u \in \Ck{1}(\R_+, \cW)$, a solution to the heat equation 
$\partial_t u = \Delta_{g,\cW}u$ with initial data $u_0 \in \Lp{2}(M)$, we have
\[
u(t,x) = \int_{M} \hk^{g,\cW}_t(x,y) u_0(y)\ d\omega_g(y)
\]
and
$\lim_{t \to 0} \hk_t^{g,\cW}(\mdot,y) \to \delta_y$ in the sense of distributions where $\delta_y$ is the Dirac-delta distribution at $y$.
\end{defn}

In the case of a smooth metric $g$, a typical construction for
the heat kernel is to construct the so-called \emph{minimal} heat kernel.
This is done by taking smooth domains $\Omega_j$, 
each of which are precompact and satisfying $\close{\Omega_j} \subset \Omega_{j+1}$. 
Inside each domain, one can solve the Dirichlet problem
to obtain heat kernels $\hk^{g,j}_{t}$, each of which
satisfies 
$$ \int_{M} \hk^{g,j}(x,y)\ d\omega_g(y) < 1.$$
Then, one can make sense of the limit $\lim_{j \to \infty} \hk^{g,j}(x,y)$
in the compact-open topology to obtain a heat kernel $\hk^{\min}_t$.
See Chapter VIII in \cite{Chavel} for the details
of this construction.

We refrain from considering this approach in the rough setting 
for the reason that this object may fail to be unique (uniqueness
is known for smooth complete $g$ with 
uniform lower bounds on Ricci curvature) and more seriously, the following
Varadhan's asymptotics may fail: 
$d^2(x,y) = \lim_{t \to 0} 4t \log \hk^{\min}_t(x,y)$.

Consequently, we can instead consider heat kernels
associated to the operator $\Delta_{g,\cW}$ in $\Lp{2}(M)$
via the Riesz representation theorem as described in Theorem 
5.2.1 in \cite{Davies}. As aforementioned, in the smooth case,
at least for the heat kernel corresponding to the 
Dirichlet and Neumann Laplacian, we obtain the desired  Varadhan's
asymptotics. See the discussion on page 107 in \cite{ERS} for details.

There is a large class of rough metrics
for which we know that the heat kernel exists
and for which Varadhan's formula holds. 
Take $h$ smooth and let $\psi:M \to M$
be a Lipeomorphism (that is, a locally bi-Lipschitz map).
Define $g = \pullb{\psi}h$, then, $g$ 
also induces a length structure with 
distance given by $d_g(x,y) = d_h(\psi(x), \psi(y))$.
Moreover, a calculation gives
that $\hk^{g}_t(x,y) = \hk^{h}_t(\psi(x), \psi(y))$
(or more generally for $\hk^{g,\cW}_t(x,y)$). 
On combining the fact that $\hk^h_t$ satisfies
Varadhan's formula, it is easy to see
so does $\hk^{g}_t$. In fact, we 
note that $\hk^g_t \in \Ck{0,1}(M \times M)$,
which is of higher regularity than we
obtain for general rough metrics.
More generally, \cite{Norris} constructs
distances on Lipschitz manifolds with our
notion of rough metrics for which this 
formula holds, and \cite{ERS} examines
this problem in even greater generality.

In what follows, we adapt the key idea in 
the proof of Theorem 5.2.1 in \cite{Davies}
to our setting. Due to the more general 
nature of our problem, we are forced to 
establish a number of a priori facts which 
are more or less immediate in the smooth case.
What we present here was initially adapted from \cite{SC} 
for the compact case in \cite{BCont}, but 
the analysis we present here shows it can be made to work more generally. 

The fundamental estimate we
require is the following \emph{weak Harnack-type inequality} for positive solutions of the heat
equation.
\begin{itemize}
\item[] At each $x \in M$ and $t > 0$, there exists a precompact
open set $U_x$, a $\delta_t \geq 0$ and a constant $C(t, U_x)  > 0$ such that
\begin{equation}
\tag{H}
\label{eq:Harnack}
\esssup_{y \in U_x} u(y,t)
	\leq C(t, U_x)  \essinf_{y \in U_x} u(y, t + \delta_t)
\end{equation}
\end{itemize}

In the compact case, even for a
rough metric, such an estimate can be
obtained with $C(t,U_x)$ precisely quantified
in terms of the curvature of a nearby smooth metric.
However, the key observation to pass from the compact analysis
to the general setting we present here was to note that
the estimates we require are purely local. In 
what is to follow, we will see that the constant can be
crude, it simply allows us to assert the existence and regularity
of $\hk^{g,\cW}_t$, but the finer properties can still 
be extracted by operator theory, in particular, from the 
fact that $t \mapsto e^{-t\Delta_{g,\cW}}$ is a semigroup.

\begin{prop}
\label{prop:MainRed}
Suppose that \eqref{eq:Harnack} holds. Then, the heat
kernel $\hk^{g,\cW}_{\mdot}: \R_{+} \times M \times M \to \R$
exists. Moreover, for every $t > 0$ and almost-every $y \in M$,  $x\mapsto \hk^{g,\cW}_t(x,y) > 0$ for almost-every $x \in M$.
\end{prop} 
\begin{proof}
We outline the steps of the construction of the heat kernel
noting that the pointwise expressions from here on should be
understood in an almost-everywhere sense.
\begin{enumerate}[(i)]
\item For $u \in \Ck{\infty}(M)$, it is readily verified that 
	$\modulus{ \nabla \modulus{u}} \leq \modulus{ \nabla u}.$
	In particular, this means that 
	for $u \in \Ck{\infty}(M) \intersect \cW$,
	the inequality
	$\norm{\nabla \modulus{u}} \leq \norm{\nabla u}$ holds
	and therefore, whenever $u \in \cW$
	we obtain $\modulus{u} \in \cW$ 
	with this estimate for such a $u$.

\item 	By construction $\dom(\sqrt{\Delta_{g, \cW}}) = \cW$
	and $\norm{\sqrt{\Delta_{g,\cW}} f} = \norm{\nabla f}$
	for all $f \in \cW$ and therefore, 
	$\norm{\sqrt{\Delta_{g, \cW}} \modulus{u}} 
		\leq \norm{\sqrt{\Delta_{g, \cW}} u}$. 
	By the Beurling-Deny condition (c.f.  Corollary 2.18(2) in  \cite{El-Maati}), 
	this yields that for $\Lp{2}(M) \ni f \geq 0$
	we have that  $e^{-t \Delta_{g,\cW}} f \geq 0$. That is, the semigroup 
	$e^{-t \Delta_{g,\cW}}$ is (weakly) positive preserving.

\item Now, let $f \in \Lp{2}(M)$, and write
	$f = f_+ - f_-$, where $f_{\pm} = \max\set{0,\pm f}$.
	It is clear that $f_{\pm} \in \Lp{2}(M)$
	and that $\modulus{f} = f_+ + f_-$.
	By the fact that we have shown $e^{-t\Delta_{g,\cW}}$ 
	is positive preserving, this means that
	for $u_{\pm} (x,t) = e^{-t\Delta_{g,\cW}} f_{\pm} \geq 0$	
	and we have that $u(x,t) = e^{-t \Delta_{g,\cW}} f = u_+(x,t) - u_-(x,t)$.

\item  Since $u_{\pm}(x,t) \geq 0$ are positive solutions,
	using inequality \eqref{eq:Harnack}, we have a precompact $U_x$ 
	with 
	\begin{align*} 
	\modulus{e^{-t \Delta_{g,\cW}}f(x)}  &= \modulus{u(x,t)} = u_+(x,t) + u_-(x,t) \\
		&\leq C(t,U_x) (u_+(y,t+\delta_t) + u_-(y,t+\delta_t)) \\
		&= C(t,U_x) \modulus{u(y,t + \delta_t)}
		= C(t,U_x)  \modulus{e^{-(t + \delta_t) \Delta_{g,\cW}}f(y)}
	\end{align*} 
	for almost-every $y \in U_x$.
	On integrating both sides over $U_x$ (which has $\omega_g(U_x) < \infty$
	by precompactness and the properties of the measure $\omega_g$) with respect
	to the variable $y$,
	we have that
	\begin{align*} 
	\modulus{e^{-t \Delta_{g,\cW}}f(x)} 
		&\leq \frac{C(t,U_x))}{\omega_g(U_x)} \int_{U_x} \modulus{e^{-(t + \delta_t) \Delta_{g,\cW}}f(y)}\ d\omega_g(y) \\
		&\leq \frac{C(t,U_x)}{\omega_g(U_x)^{\frac{1}{2}}} \norm{e^{-(t + \delta_t) \Delta_{g,\cW}f}},
	\end{align*}
	where the second inequality follows from the Cauchy-Schwartz
	inequality and since $\norm{\cdot}_{\Lp{2}(U_x)} \leq \norm{\cdot}_{\Lp{2}(M)}$.

\item By the self-adjointness of $\Delta_{g,\cW}$, via functional calculus 
	we obtain that $\norm{e^{-s\Delta_{g,\cW}}f} \leq \norm{f}$
	uniformly in $s$  and therefore, we obtain that for each fixed \((t, x)\),
	$$\modulus{e^{-t \Delta_{g,\cW}}f(x)} \leq 
		\frac{C(t,U_x)}{\omega_g(U_x)^{\frac{1}{2}}} \norm{f}.$$
	This exactly says that $f \mapsto (e^{-t\Delta_{g,\cW}}f)(x) \in \adj{(\Lp{2}(M))}$
	for every $f \in \Lp{2}(M)$ and therefore, by the Riesz Representation theorem,
	we obtain, for each fixed \((t, x)\), an $a_{t,x} \in \Lp{2}(M)$ such that
	$$ e^{-t \Delta_{g,\cW}}f(x) = \inprod{a_{t,x},f} = \int_{M} a_{t,x}(y)f(y)\ d\omega_g(y)$$
        for every \(f \in \Lp{2}\). Because \(t \mapsto e^{-\Delta t} f\) is a continuous semigroup mapping 
	\(\Lp{2}(M) \to \Lp{2}(M)\), we have that \(x \mapsto a_{t,x}\) is measurable for each $t > 0$.
	In fact, \((t, x) \mapsto a_{t,x}\) is jointly measurable by Lemma 4.51 in \cite{AB}.

\item We now define the symmetric function,
  $$\hk^{g,\cW}_t(x,y) := \inprod{a_{\frac{t}{2}, x}, a_{\frac{t}{2}, y}}$$
        We claim that \(\hk_t^{g,\cW}\) is a heat kernel: we use the by the semigroup property and self adjointness of \(e^{-t\Delta_{g,\cW}}\). For clarity, let us write \(\inprod{f(\cdot), g(\cdot)}\) to denote that the integration is over the \(\cdot\) variable. Then we have
	\[
        \begin{split}
        (e^{-t \Delta_{g,\cW}}f) (x) &= e^{-\frac{t}{2} \Delta_{g,\cW}} (e^{-\frac{t}{2} \Delta_{g,\cW}} f) (x) = \inprod{a_{t,x}, e^{-\frac{t}{2} \Delta_{g,\cW}} f} \\
        &= \inprod{e^{-\frac{t}{2} \Delta_{g,\cW}} a_{t,x}, f} = \inprod{\inprod{a_{t,\cdot}, a_{t,x}}, f} \\
        &= \inprod{\hk^{g,\cW}_t(x,(\cdot)), f}.
        \end{split}
        \]
        Next, since $e^{-t \Delta_{g,\cW}}$	is positive preserving, $\hk_t^{g,\cW}$ is real-valued. Moreover, since $\lim_{t\to 0} e^{-t\Delta_{g,\cW}}f = f$ in $\Lp{2}(M)$, we have that $\lim_{t \to 0} \hk^{g,\cW}_t(\mdot,y) \to  \delta(y)$ in the sense of distributions.

        By definition of the semigroup and $a_{t,x}$, we have that $x\mapsto e^{\Delta_{g,\cW}}f(x) = \inprod{a_{t,x},f} \in \Lp{2}(M)$
	and hence measurable in $x$ for each $f$. 
	On choosing \(f = a_{t,y}\), we obtain 
	measurability in \(x\) and symmetry of \(\rho_t\) yields measurability in \(y\),
	which shows that \(\rho_t(x, y)\) is measurable separately in both \(x\) and \(y\). 

\item To show heat-kernel is strictly positive, we first localise the semi-group.
	Let $\Omega \subset M$ be an open connected subset with compact closure and piecewise smooth boundary  and let $f \in \Lp{2}(\Omega)$.
	Observing that $\SobH[0]{1}(\Omega) \embed \cW$ embeds isometrically as a closed subspace via extension by zero outside of $\Omega$, we consider \(\SobH[0]{1}(\Omega)\) as a closed subspace of $\cW$. Likewise for $\Lp{2}(\Omega) \embed \Lp{2}(M)$, in which case the $\Lp{2}(M)$-orthogonal projection onto $\Lp{2}(\Omega)$ is just $f \mapsto \chi_{\Omega} f$.

        When $u \in \SobH[0]{1}(\Omega)$, we have that $\modulus{u} \in \SobH[0]{1}(\Omega) \subset \cW$ and that $v \in \cW$ with $\modulus{v} \leq \modulus{u}$ implies that $v \sgn u \in \SobH[0]{1}(\Omega)$,
	where $\sgn u = u/\modulus{u}$ for $u(x) \neq 0$ and $0$ for $u(x) = 0$. That is, $\SobH[0]{1}(\Omega)$ is a \emph{ideal} in $\cW$. Let $\Delta_{g,\SobH[0]{1}(\Omega)}$ denote the associated Dirichlet Laplacian inside $\Omega$ with respect to the energy $\inprod{\nabla u, \nabla v}_{\Lp{2}(\Omega, g)}$. Then , we see that (vi) in Theorem 4.1 in \cite{MVV} by Manavi, Vogt and Voigt is satisfied, hence we obtain that (ii) of the same theorem holds. That is,
 	$$ \modulus{e^{-t\Delta_{g, \SobH[0]{1}(\Omega)}} f} \leq e^{-t \Delta_{g, \cW}} \modulus{f}$$
	for all $f \in \Lp{2}(\Omega,g)$.

        Therefore, when $f \in \Lp{2}(\Omega,g)$ with $f \geq 0$, we have by using the non-negativity of $\Delta_{g, \SobH[0]{1}(\Omega)}$ (by (i) of this proposition) that
	$$ \modulus{e^{-t\Delta_{g, \SobH[0]{1}(\Omega)}} (\chi_{\Omega} f)} \leq e^{-t \Delta_{g, \cW}} \modulus{f}.$$

\item Now, we show that the localised semi-group is strictly positive: $e^{-t\Delta_{g,\SobH[0]{1}(\Omega)}}f (x) > 0$ for $x$-a.e. in $\Omega$ and $t>0$ whenever $f \geq 0$ but $f \neq 0$.
	That is essentially given by the proof of Theorem 4.5 in \cite{El-Maati}, which is done for a quadratic form in Euclidean space, but works in this context with minimal modification.
	Note first that by the compactness of $\close{\Omega}$, there exists a smooth metric $h$ on $\Omega$ that is quasi-isometric to $g$ on $\Omega$   (but not necessarily all of $M$) so that 
	$$\inprod{\nabla u, \nabla v}_{\Lp{2}(\Omega,g)} = \int_{\Omega} h((\sqrt{\det B}) B \nabla u, \nabla v)\ d\omega_h = \inprod{(\sqrt{\det B})B \nabla u, \nabla v}_{\Lp{2}(\Omega,h)},
	$$
	where $B$ are symmetric, bounded, measurable coefficients such that $h(Bu,v) = g(u,v)$.

	Corollary 2.11 in \cite{El-Maati} states the strict positivity of the semigroup  holds if and only if every subset $\Omega' \subset \Omega$ with the property that $\chi_{\Omega'} u \in \Sob[0]{1}(\Omega)$ for every $u \in \Sob[0]{1}(\Omega)$ must satisfy either $\omega_h(\Omega') = 0$ or $\omega_h(\Omega\setminus \Omega') = 0$. 
	For contradiction, suppose there exists such an $\Omega'$ with $\chi_{\Omega'} u \in \SobH[0]{1}(\Omega)$ for every $u \in \Sob[0]{1}(\Omega)$ but $\omega_h(\Omega') > 0$ and $\omega_h(\Omega \setminus \Omega') > 0$.

        Now, there exists $x_0 \in \Omega$ such that for all $\eta > 0$,
	$$\omega_h(B(x_0, \eta) \intersect \Omega') > 0\quad\text{and}\quad\omega_h(B(x_0, \eta) \intersect (\Omega\setminus\Omega'))> 0.$$
	If this weren't the case, then for all $x \in \Omega$, there exists $\eta_x > 0$ such that $\omega_h(B(x, \eta_x) \intersect \Omega) = 0$ or $\omega_h(B(x_0, \eta_x) \intersect (\Omega\setminus\Omega')) = 0$. In this case, we write 
	\begin{align*} 
	\Omega_1 &= \Union \set{B(x,\eta_x): \omega_h(B(x, \eta_x) \intersect \Omega') = 0}\ \text{and}\\
	\Omega_2 &= \Union \set{B(x,\eta_x): \omega_h(B(x, \eta_x) \intersect(\Omega \setminus \Omega') = 0}
	\end{align*}
	to obtain two open sets which satisfy $\Omega = \Omega_1 \union \Omega_2$ and $\Omega_1 \intersect \Omega_2 = \emptyset$ (metric balls have positive measure). Since $\Omega$ is connected, either $\Omega_1 = \Omega$ or $\Omega_2 = \Omega$. If $\Omega = \Omega_1$, then we would conclude that $\omega_h(\Omega') = 0$ contrary to assumption. Similarly if $\Omega = \Omega_2$ then $\omega_h(\Omega \setminus \Omega') = 0$ contrary to assumption.

	Having established the existence of $x_0$, given any $\eta > 0$ sufficiently small so that $B(x_0,\eta) \subset \Omega$, let $u \in \Ck[c]{\infty}(\Omega)$ be such that $u(x) = 1$ on $B(x_0, \eta/2)$ and $0$ outside $B(x_0,\eta)$. By assumption, $\chi_{\Omega'} u \in \SobH[0]{1}(\Omega)$ and  given that $\nabla$ is a local operator, $\nabla (\chi_{\Omega'} u)(x) = (\nabla u)(x)$ for $x$-a.e. in $\Omega'$ and $0$ otherwise. Then we obtain that
	$$\chi_{\Omega'}\nabla u = \nabla(\chi_{\Omega'} u)$$
	inside $\Omega$.
	Combining this with the assumption that the  closure of $\Omega$ is compact with piecewise smooth boundary, we obtain that $\chi_{\Omega'} u \in \Sob{1,p}(\Omega)$ for all $p \in [1,\infty]$. 
	Also since $\partial \Omega$ is piecewise smooth, from the Sobolev embedding theorem for manifolds with piecewise smooth boundary (c.f. Theorem 2.34 in \cite{Aubin}), we get that $\chi_{\Omega'} u = v$ for some continuous $v$.

	Then, for all $x \in B(x_0,\eta) \intersect \Omega'$ and $y \in B(x_0,\eta) \intersect (\Omega \setminus \Omega')$, 
	$$1 = \chi_{\Omega'}(x)u(x) = \chi_{\Omega'}(x) u(x) - \chi_{\Omega'}(y)u(y) = \modulus{v(x) - v(y)}$$ 
	contradicting that $v$ is continuous on
        $$B(x_0, \eta) = (B(x_0,\eta) \intersect \Omega') \bigcup (B(x_0,\eta) \intersect (\Omega \setminus \Omega')).$$

\item Next we show that $\omega_g\set{y \in \Omega: \hk^{g,\cW}_{t}(x,y) = 0\ \text{x-a.e. in}\ \Omega} = 0$ for every $t>0$. 
	That is, we want to show that $x \mapsto \hk^{g,\cW}_{t}(x,y)$ is not identically zero for almost every $y \in \Omega$.
	To obtain a contradiction, suppose not.
	That is, for some $t > 0$, we have that $\omega_g\set{y \in \Omega: \hk^{g,\cW}_{t}(x,y) = 0\ \text{x-a.e. in}\ \Omega} > 0$.
	Since $M$ is $\sigma$-finite, we can find a set $P \subset \Omega$ with $\omega_g(P) > 0$ and for $y \in P$, we have that $\hk^{g,\cW}_t(x,y) = 0$ for $x$-a.e. in $\Omega$. 
	Then for any $y \in P$ and $f \in \Lp{2}(M, g)$ vanishing outside of $P$, we have
	$$ (e^{-t\Delta_{g,\cW}} f)(y) = \int_{P} \hk_t^{g,\cW}(x,y) f(x)\ d\omega_g(x) = 0.$$
	In particular, setting $f = \chi_{P}$, we obtain
	$$ 0 = \inprod{e^{-t\Delta_{g,\cW}} \chi_{P}, \chi_{P}} = \norm{  e^{-\frac{t}{2}\Delta_{g,\cW}} \chi_{P}}_{\Lp{2}}.$$ 
	That is, $e^{-\frac{t}{2}\Delta_{g,\cW}} \chi_P = 0$ in $\Lp{2}(M, g)$ and iterating this procedure in this way, we obtain that
	$ e^{-\frac{t}{2^i}\Delta_{g,\cW}}\chi_P = 0$ for all $i \in \Na$.
	Therefore, $\lim_{i \to \infty} e^{-\frac{t}{2^i}\Delta_{g,\cW}} \chi_P = 0$ in $\Lp{2}(M, g)$.
	However, since $\Delta_{g,\cW}$ is non-negative self-adjoint, operator theory yields that $\chi_P = \lim_{i \to \infty} e^{-\frac{t}{2^i}\Delta_{g,\cW}}\chi_P$.
        But \(\chi_P \equiv 1\) on the positive measure set $P$ contradicting that $\lim_{i \to \infty} e^{-\frac{t}{2^i}\Delta_{g,\cW}}\chi_P = 0$ in $\Lp{2}(M, g)$.
	
\item Finally, we obtain the strict positivity of the heat kernel. By (ii) and (vi) we have for any $f \in \Lp{2}(M, g)$ with $f \geq 0$,
        $$\inprod{\hk^{g,\cW}_t(x,\cdot), f} = (e^{-t \Delta_{g,\cW}}f) (x) \geq 0$$
        and hence $\hk_t^{g,\cW} \geq 0$ for $x$-a.e. in $M$.

        Suppose for a contradiction to $\hk_t^{g,\cW} > 0$ for all $t > 0$ and $x$-a.e. in $M$ that there exists positive measure sets $M'$ and $M''$ and a $t > 0$ such that for $y$-a.e. in $M'$, $\hk^{g,\cW}_t(x,y) = 0$ for $x$ a.e. in $M''$.
	Since $M = \union_{j=1}^\infty \Omega_j$, where each $\Omega_j$ has piecewise smooth boundary and has compact closure, 
	there exists a $K$ such that for $\Omega = \union_{j=1}^K \Omega_j$ (a set with piecewise smooth boundary) we have $\omega_g(M' \intersect \Omega) > 0$ and $\omega_g(M'' \intersect \Omega) > 0$.
        
	But then applying (ix) and (vii) on $\Omega$ to $f = \chi_{\Omega} \hk^{g,\cW}_{t/2} \geq 0$ we obtain that for $y$-a.e. in $\Omega$,
	$$0 < e^{-\frac{t}{2}} \Delta_{g,\SobH[0]{1}(\Omega)}(\chi_{\Omega} \hk^{g,\cW}_{\frac{t}{2}}(\cdot,y))(x) \leq e^{-\frac{t}{2}} \Delta_{g,\cW}\hk^{g,\cW}_{\frac{t}{2}}(\cdot,y))(x)  = \hk^{g,\cW}_{t}(x,y)$$
	for almost every $x$ in $\Omega$ contradicting that $\omega_g(M'') > 0$.
        This shows that for almost every $y \in M$, $x \mapsto \hk^{g,\cW}_t(x,y) > 0$ for $x$-a.e.
	\qedhere
\end{enumerate} 
\end{proof}

\begin{rem}
We note that the Harnack-type estimate \eqref{eq:Harnack} that we assume is very 
weak, i.e., it is not defined on cylinders or even 
parabolic cylinders. The proof shows that the existence 
of the heat kernel only requires such an estimate. 
As we shall see in the next section, a stronger
estimate is required for regularity of solutions. 
\end{rem}

\section{The Harnack inequality}
\label{sec:harnack}

In this section, we demonstrate that a
stronger form of the  Harnack 
estimate than \eqref{eq:Harnack} holds for rough 
metrics for which  \eqref{eq:Harnack} is a consequence.
To describe this estimate, fix a collection of 
locally comparable charts $\psi_x:V_x \to \R^n$ 
for each $x \in M$ such that $\psi_x(V_x) = B_{r_x}$,
where $B_{r_x}$ is a ball of radius $r_x > 0$. 
Then, let
$\delta_x(\cdot, \cdot) = \pullb{\psi_x}\delta(\cdot, \cdot)_{\R^n}$
denote the pullback metric with $\Leb_x = \pullb{\psi_x}\Leb$
and $d_x(y,y') = \modulus{\psi_x(y') - \psi_x(y)}_{\R^n}$.
Recall that there is a $C_x \geq 1$
with $C_x^{-1}  \modulus{u}_{\delta_x} \leq \modulus{u}_g \leq C_x \modulus{u}_{\delta_x}$
almost-everywhere in $V_x$.

Fix $0 < \kappa < \tau < \infty$ and an $0 < \epsilon < \kappa$,
and for $x \in M$ and $t \in (\kappa, \tau)$, define
\begin{equation}
\label{eq:paradef}
\begin{split} 
Q^-_{(x,t)}(\kappa, \tau, \epsilon) &= \set{(y,s) \in V_x \times (\kappa,\tau): s \in \cbrac{t - \frac{3}{4}\epsilon^2, t - \frac{1}{4} \epsilon^2},\ 
	d_x(x,y) < \frac{1}{2} \epsilon} \\
Q^+_{(x,t)}(\kappa, \tau, \epsilon) &= \set{(y,s) \in V_x \times (\kappa,\tau): s \in \cbrac{t + \frac{3}{4}\epsilon^2, t + \epsilon^2},\ 
	d_x(x,y) < \frac{1}{2} \epsilon}
\end{split}
\end{equation}

Let $g(u,v) = \delta_x(A_x u,v)$
and $a_x = \sqrt{\det A_x}$ so that
$\omega_x(y) = a_x(y)d\Leb_x(y)$, and note
that:
\begin{equation}
\label{eq:constbd}
\begin{split}
&C_x^{-2} \modulus{u}_{\delta_x}^2 \leq \delta_x(A_x u, u) \leq C_x^2 \modulus{u}_{\delta_x}^2 \\
&C_x^{-\frac{n}{2}} \leq a_x \leq C_x^{\frac{n}{2}}.
\end{split}
\end{equation}
See Section 3.3 in \cite{BRough} for details. 

\begin{lemma}[Localisation Lemma]
\label{eq:soln1}
Let $u(\cdot, \cdot): M \times \R_+$  be a a solution to the $\Delta^{g,\cW}$ heat equation.
Then, for all $v \in \SobH[0]{1}(V_x,g)$, 
	\begin{equation} 
	\label{eq:testfn1}	 
	\inprod{a_x \partial_t u, v}_{\Lp{2}(V_x, \delta_x)} 
		= \inprod{ B_x \nabla u, \nabla v}_{\Lp{2}(V_x, \delta_x)},
	\end{equation}
where $B_x = a_x A_x$. Moreover,
\begin{equation}
\label{eq:coeff}
C_x^{-2} a_x \modulus{\xi}_{\delta_x}^2 \leq \delta_x(B_x \xi, \xi) 
	\leq C_x^2 a_x \modulus{\xi}_{\delta_x}^2
\end{equation} 
for every $\xi \in T_xM$ and almost everywhere in $V_x$. 
\end{lemma}
\begin{proof}
Since for $t > 0$ the solution $u \in \dom(\Delta_{g,\cW}) \subset \Lp{2}(M)$, we
note that in particular, for any $v \in \Ck[c]{\infty}(M)$, 
$\inprod{\partial_t u, v} = \inprod{\nabla u, \nabla v}$.
Choosing $\spt v \subset V_x$, we note that $\inprod{\cdot, v} = \inprod{\cdot, v}_{\Lp{2}(V_x,g)}$
and hence,
$$\inprod{\partial_t u, v} = \int_{V_x} \partial_t u\conj{v}\ d\omega_g 
	= \int_{V_x} a_x \partial_t  u \conj{v}\ d\Leb_x
	= \inprod{\partial_t (a_x u), v}_{\Lp{2}(U_x, \delta_x)},$$
where the last equality follows from the fact that 
$a_x(y) \partial_t u(y,t) = \partial_t (a_x(y)u(y,t).$
Similarly,
$$\inprod{\nabla u, \nabla v} 
	= \int_{V_x} g(\nabla u, \nabla v)\ d\omega_g
	= \int_{V_x} \delta_x(a_x A_x \nabla u, \nabla v)\ d\Leb_x
	= \inprod{B_x \nabla u, \nabla v}_{\Lp{2}(U_x, \delta_x)}.$$
The estimate on $B_x$ follows immediately from
\eqref{eq:constbd}.
\end{proof} 

The Harnack inequality we prove is the following. We provide two proofs of this result, 
the first proof simply on noting that we can deduce this from noting that the results
of \cite{SC} localise, and the second from degenerate parabolic equation results of \cite{CS}.

\begin{thm}
\label{thm:Harnack}
Let $u(x,t) \geq 0$ be a
solution to the $\Delta_{g,\cW}$ heat equation 
that is non-negative in $(\kappa, \tau)$. 
Then, there exists $\gamma = \gamma(n, C_x, \kappa, \tau) > 0$ such that 
$$ \sup_{(y,s) \in Q^-_{(x,t)}(\kappa, \tau,\epsilon)} u(y, s) 
	\leq \gamma \inf_{(y,s) \in Q^+_{(x,t)}(\kappa,\tau,\epsilon)} u(y,s)$$
for all $\epsilon < \min\set{ \sqrt{t - \kappa}, \sqrt{\tau - t}, r_x}$.
\end{thm}

\begin{proof}[Proof using \cite{SC}]
By the use of Lemma \ref{eq:soln1}, which illustrates that solution $u$ of the $\Delta_{g,\cW}$ heat equation is a distributional solution, we can expect to use the methods of Moser's parabolic Harnack inequality \cite{MR0159139,MR0288405}. However, Lemma \ref{eq:soln1} shows that we have a heat equation in a weighted Sobolev space, weighted by $a_x$, and therefore, we cannot immediately apply Moser's results. But, the factor $a_x > 0$ almost-everywhere, time independent and measurable, which is precisely the situation described in \cite[Section 4]{SC}.  Unfortunately, the results of \cite{SC} demand that the rough metric is globally comparable to a complete, smooth metric of Ricci curvature bounded from below. However, upon more careful inspection of \cite{SC}, we can see that the estimate is entirely local and it is only through the constants that the curvature bounds and global comparability enter. In our situation, we do not obtain uniform, global constants (and indeed do not expect them at this level of generality), but such control is not necessary for our applications. Thus, on noting that the results in \cite{SC} localise, we can apply Theorem 5.3 in \cite{SC} which is precisely our Harnack inequality stated in the theorem.
\end{proof}

In the compact setting, this theorem 
is immediate from Theorem 5.3 in \cite{SC} without having to note
its local nature. There, 
he proves such estimates for general operators
$L = -a \divv A \nabla$, where $a \in \Lp{\infty}(M)$
and $A \in \Lp{\infty}(T^{1,1}M)$ and symmetric
for smooth metric $h$ with a uniform lower bound
on Ricci curvature.
The key to note is that a rough metric in the 
compact setting is globally comparable to a smooth one, 
i.e., there exists a global constant $C \geq 1$ 
and a 
such that
$$ C^{-1} \modulus{u}_{h} \leq \modulus{u}_g \leq C \modulus{u}_{h}$$
for $u \in T_x M$ for almost-every $x \in M$.
By the virtue of compactness, it is immediate
that the Ricci curvature of $h$ is bounded below
by a uniform constant. 
In fact, this procedure also works 
for rough metrics $g$ which 
are uniformly close to some complete metric $h$ with 
lower bound on Ricci curvature in the sense
we have just written, rather than having to assume compactness.

We emphasise that the goal in \cite{SC}
is to quantify the constants $C(t,U_x)$
appearing in Theorem \ref{thm:Harnack}. In what we present here, we do not have
the ability to control these constants. However, as
Proposition \ref{prop:MainRed} illustrates, this is not necessary 
in order to obtain the existence of the heat kernel.
Quantifying such estimates is, however, still 
extremely important and useful when it can be done, 
since this can be used to obtain better global regularity for
solutions than we do here.

We now give a proof using the results of \cite{CS}.
From here on, we assume that the dimension of $M$ is $3$ or greater. 

Fix $0 < \kappa < \tau < \infty$ and 
define $u_\kappa(x,t) = u(x,t+\kappa)$. Then we note the following. 
\begin{lemma}
\label{lem:shiftsol}
We have that: 
\begin{enumerate}[(i)] 
\item $u_\kappa$ is a solution to the $\Delta_{g,\cW}$ heat equation, 
\item there is a constant $C_\kappa \geq 0$
	which depends on $\kappa$ 
	such that $\norm{\nabla u_\kappa} \leq C_\kappa$.
\end{enumerate} 
\end{lemma}
\begin{proof}
Note that by definition, and using the semigroup property, 
$$u_\kappa(x,t) = e^{-(t + \kappa)\Delta_{g,\cW}} u_0 
	=e^{-t \Delta_{g,\cW}}e^{-\kappa\Delta_{g,\cW}}u_0
	=e^{-t \Delta_{g,\cW}}u(x,\kappa).$$ 
Therefore, it is immediate that $u_\kappa$ is a solution
to the heat equation. 

To see the bound, we compute: 
\begin{multline*} 
\norm{ \nabla u_\kappa(\cdot, t)} 
	= \norm{ \sqrt{\Delta_{g,\cW}} u_\kappa(\cdot, t)} 
	= \norm{ \sqrt{\Delta_{g,\cW}} e^{-t\Delta_{g,\cW}} e^{-\kappa \Delta_{g,\cW}}u_0} \\
	= \norm{ e^{-t \Delta_{g,\cW}} \sqrt{ \Delta_{g,\cW}} e^{-\kappa \Delta_{g,\cW}}u_0}
	\leq \norm{\sqrt{ \Delta_{g,\cW}} e^{-\kappa \Delta_{g,\cW}}u_0},
\end{multline*}
where the penultimate equality follows from the fact that
$\ran(e^{-\kappa \Delta_{g,\cW}}) \subset \dom(\Delta_{g,\cW}^\alpha)$
for all $\alpha > 0$ and via functional calculus.
The proof is complete on setting the constant 
$C(\kappa) = \norm{\sqrt{ \Delta_{g,\cW}} e^{-\kappa \Delta_{g,\cW}}u_0}$.
\end{proof}

In what is to follow, we shall require some facts about 
Lebesgue and Sobolev space theory for functions valued 
in Banach spaces. While the book \cite{CH} discusses these issues
in detail, an excellent overview of this topic
is contained in the thesis \cite{Kreuter}.

As in \cite{CS}, for $T > 0$, define 
$$ W(T) = \set{ w \in \Lp{2}( (0,T); \SobH[0]{1}(V_x,g)): \partial_t w \in \Lp{2}((0,T); \Lp{2}(V_x, g))},$$
and let $W_0(T) = \set{ w \in W: w(0) = w(T) = 0}.$

\begin{lemma}
\label{Lem:testfn2} 
The solution $u_\kappa \in \Lp{2}((0, \tau - \kappa); \SobH{1}(V_x, g))$
and is a weak solution in the following sense:
\begin{equation}
\label{eq:testfn2} 
\int_{0}^{\tau - \kappa} \inprod{a_x u_\kappa(t), \partial_t w(t)}_{\Lp{2}(V_x,\delta_x)}\ dt 
		= -\int_{0}^{\tau - \kappa } \inprod{B_x \nabla u_\kappa(t), \nabla w(t)}_{\Lp{2}(V_x, \delta_x)}\ dt
\end{equation}
for all $w \in W_0(\tau - \kappa)$.
\end{lemma}
\begin{proof}
Fix $w \in W_0$ and from \eqref{eq:soln1}, we have that
$$\inprod{a_x \partial_t u_\kappa (t), w(t)}_{\Lp{2}(V_x, \delta_x)} 
		= \inprod{ B_x \nabla u_\kappa(t), \nabla v(t)}_{\Lp{2}(V_x, \delta_x)}$$
for every $t \in (0, \tau - \kappa)$.
A calculation similar to that in Lemma \ref{lem:shiftsol} 
yields that $\norm{ \partial_t u_\kappa} \leq C_\kappa'$
for all $t \in (0, \tau - \kappa)$
on noting that $\partial_t u_\kappa = \Delta_{g,\cW} u_\kappa$.
This, along with the bound in Lemma \ref{lem:shiftsol}
shows that  $u_\kappa \in \Lp{2}((0, \tau - \kappa); \SobH{1}(V_x, g))$
and that $\partial_t u_\kappa \in \Lp{2}((0,\tau - \kappa); \Lp{2}(V_x, g))$. 

Moreover, note that $W(\tau- \kappa) \subset \SobH{1}((0,\tau-\kappa); \Lp{2}(V_x,g))$
and since in particular $w(0) = w(\tau - \kappa) = 0$, by the fundamental theorem of calculus in the Banach 
valued setting by Proposition 1.4.29, Corollary 1.4.31 and Corollary 1.4.37 in  \cite{CH}, we obtain that
$$\int_{0}^{\tau - \kappa} \inprod{ \partial_t(a_x u_\kappa(t)), w(t)}_{\Lp{2}(V_x,\delta_x)}\ dt
	= - \int_0^{\tau - \kappa} \inprod{a_x u_\kappa(t), \partial_t w(t)}_{\Lp{2}(V_x, \delta_x)}\ dt.$$
Also, the integral on the left is equal to:
$$ \int_0^{\tau - \kappa}  \inprod{ B_x \nabla u_\kappa(t), \nabla v(t)}_{\Lp{2}(V_x, \delta_x)}$$
and since  $\nabla u_\kappa \in \Lp{2}((0, \tau - \kappa); \Lp{2}(V_x,g))$,
the proof is complete.
\end{proof} 

With these three lemmas in hand, we prove the main theorem.

\begin{proof}[Proof of Theorem \ref{thm:Harnack} for $\dim M \geq 3$ using \cite{CS}]
Let $\epsilon > 0$ be sufficiently small to 
be determined later. 
Fix $0 < \kappa < \tau < \infty$ and suppose that $u(x,t) \geq 0$ 
is a positive solution of
the $\Delta_{g,\cW}$ in $(\kappa,\tau)$. 
Then, as before, write $u_\kappa(x,t) = u(x,t+\kappa)$.

Let $D$ be any $n$-dimensional cube inside $\psi_x(V_x)$ then 
on writing $D_x = \psi_x^{1}(D)$, we have  
$$ \int_{D_x} a_x(y)\ d\Leb_x(y) \leq C_x^{\frac{n}{2}} \Leb_x(C),
\ \text{and}\ 
\int_{D_x} \frac{1}{a_x(y)}\ d\Leb_x(y) \leq C_x^{\frac{n}{2}}.$$
Therefore, our density $a_x$ satisfies the so called $A_2$ condition:
$$c_0 := \sup_{D_x} \cbrac{ \frac{1}{\Leb_x(D_x)} \int_{D_x} a_x(y)\ d\Leb_x(y)}
\cbrac{ \frac{1}{\Leb_x(D_x)} \int_{D_x} \frac{1}{a_x(y)}\ d\Leb_x(y)} 
	\leq C_x^n.$$
Combining this with the \eqref{eq:coeff} as
well as \eqref{eq:testfn2} from Lemma \ref{Lem:testfn2} shows that $u_\kappa$
satisfies the hypotheses of Theorem 2.1 in \cite{CS} and
therefore, we obtain a $\gamma = \gamma(C_x, n, \kappa, \tau) > 0$
such that 
\begin{equation} 
\label{eq:harnack2} 
\sup_{(y,s) \in Q_{(x,t-\kappa)}^-(0, \tau - \kappa, \epsilon)} u_\kappa(y,s)
		\leq \gamma \inf_{(y,s) \in Q^+_{(x,t-\kappa)}(0, \tau - \kappa, \epsilon)} u_\kappa(y,s).
\end{equation}
Note now that setting $s' = s + \kappa$, we have that
$(y, s' - \kappa) \in Q^{\pm}_{(x,t- \kappa)}(0, \tau - \kappa, \epsilon)$
if and only if $(y, s') \in Q^{\pm}(x,t)(\kappa, \tau, \epsilon)$.
That is, \eqref{eq:harnack2} is equivalent to the statement 
in the conclusion of the theorem and hence, this concludes
the proof.

Now, to compute a bound for $\epsilon$, note that
we want to ensure 
\begin{multline*}\set{(y,s) \in V_x \times (0, \infty): |y - s| < \epsilon, |t - s| < \epsilon^2} \\
\subset \set{(y,s) \in V_x \times (\kappa, \tau): |y - s| < \epsilon, |t - s| < \epsilon^2}
\end{multline*} 
we note that we require $\kappa < t - \frac{3}{4} \epsilon^2$,
$t + \epsilon^2 < \tau$ and $\epsilon < r_x$.
Rearranging this gives the range of $\epsilon$
in the conclusion.
\end{proof} 

\section{Regularity of solutions}

An important consequence of Harnack estimates for weak 
solutions is that they yield a priori regularity estimates
for those solutions. It is classical fact how these estimates
yield regularity results. For the benefit of the reader, 
we give the following brief outline of how to obtain
regularity from the estimates in Theorem \ref{thm:Harnack}.
We follow the argument of Theorem 6.28 in 
\cite{Lieberman} by Lieberman.

\begin{prop}
Let $u$ be a solution to the $\Delta_{g,\cW}$
heat equation (not necessarily positive).
Fix $x \in M$ and $0 < t_1 < t_2 < \infty$.
Then, there exists an open set $U_x \subset V_x$ 
with $x \in U_x$
such that 
$$ \modulus{u(y,s) - u(y',s)} \leq C(n, C_x, t_1, t_2) d_x(y, y')^\alpha,$$
where 
$$ \alpha = \log_{\frac{1}{4}}\cbrac{1 - \frac{1}{\gamma}}$$
for every $s \in [t_1, t_2]$.
\end{prop}
\begin{proof}
Let 
$\kappa = \frac{t_1}{2},\ \tau = \frac{3t_2}{2},\ R_0 = \frac{1}{4} \min\set{t_1, r_x},$ fix $r < \frac{R_0}{4}$.
Define, 
\begin{align*}
&M_4 = \sup_{(y,s) \in Q^{-}_{(x,t)}(\kappa, \tau, 4r)} u(y,s) &&m_4 = \inf_{(y,s) \in Q^{-}_{(x,t)}(\kappa, \tau, 4r)} u(y,s) \\
&M_1 = \sup_{(y,s) \in Q^{-}_{(x,t)}(\kappa, \tau, r)} u(y,s) &&m_1 = \inf_{(y,s) \in Q^{-}_{(x,t)}(\kappa, \tau, r)} u(y,s), 
\end{align*} 
and note that $M_j - u$ and $u - m_j$ for $j = 1, 4$ are non-negative
solutions to the $\Delta_{g,\cW}$ heat equation.
By the choice of $R_0$,
we always have that $r < \min\set{ \sqrt{t - \kappa}, \sqrt{\tau - t}, r_x}$
for every $t \in [t_1, t_2]$. 
So, fix such a $t \in [t_1, t_2]$ we invoke Theorem \ref{thm:Harnack} and integrate,
\begin{multline*}
\iint_{Q^-_{(x,t)}(\kappa,\tau,r)} (M_4 - u)\ d\Leb_x dt 
	\leq \gamma \iint_{Q^-_{(x,t)}(\kappa,\tau,r)} \inf_{Q^+_{(x,t)}(\kappa, \tau, r)} (M_4 - u)\ d\Leb_x ds \\
	\leq \gamma (M_4 - M_1) \int_{t- \frac{3}{4}r^2}^{t - \frac{1}{4}r^2} \int_{\psi_x^{-1}(B(x,r))}\ d\Leb_x ds
	= \gamma w_n  \frac{1}{2}r^{2+n} (M_4 - M_1),$$
\end{multline*}
where $w_n$ is the constant for which $\Leb(B(x,r)) = w_n r^n$.
Note that the constant $\gamma$ is independent of $t$, and only 
dependent on $t_1$ and $t_2$ through our choice
of values for $\kappa$ and $\tau$. 
By a similar calculation, 
$$
\iint_{Q^-_{(x,t)}(\kappa,\tau,r)} (u - m_4)\ d\Leb_x dt \leq \gamma w_n \frac{1}{2}r^{2+n} (m_1 - m_4),$$
and adding these two inequalities together, 
we get that
$$
\frac{1}{2}r^{2+n} w_n (M_4 - m_4) \leq \gamma w_n \frac{1}{2}r^{2+n} (M_4 - m_4 +m_1 - M_1).$$
Rearranging, we find that 
$$
\osc_{Q^{-}_{(x,t)}(\kappa, \tau, r)} u \leq \cbrac{1 - \frac{1}{\gamma}} \osc_{Q^{-}_{(x,t)}(\kappa, \tau, 4r)} u.$$
We note that the constant $(1 - \gamma{-1}) < 1$
and therefore, it is of the right form to invoke the
standard iteration procedure given in Lemma 4.6 in \cite{Lieberman}. 
More specifically, on noting that
our oscillation estimate is of the form (4.15)'' in \cite{Lieberman},
we obtain the precise form for $\alpha$.

The passage from this to the Hölder estimate
we have noted in the conclusion is immediate. 
\end{proof}

We conclude this paper by noting that the proof of Theorem \ref{Thm:Main}
follows immediately upon collating the facts we have established
in this paper. The regularity result of global continuity then ensures all the notions of almost-everywhere become pointwise everywhere.

\bibliographystyle{amsalpha}

\begin{thebibliography}{AKM06}

\bibitem[AB06]{AB}
Charalambos~D. Aliprantis and Kim~C. Border, \emph{Infinite dimensional
  analysis: a hitchhiker's guide}, Springer, Berlin; London, 2006.

\bibitem[AKM06]{AKM2}
Andreas Axelsson, Stephen Keith, and Alan McIntosh, \emph{The {K}ato square
  root problem for mixed boundary value problems}, J. London Math. Soc. (2)
  \textbf{74} (2006), no.~1, 113--130.

\bibitem[Aub82]{Aubin}
Thierry Aubin, \emph{Nonlinear analysis on manifolds. {M}onge-{A}mp\`ere
  equations}, Grundlehren der Mathematischen Wissenschaften [Fundamental
  Principles of Mathematical Sciences], vol. 252, Springer-Verlag, New York,
  1982. \MR{681859}

\bibitem[Ban14]{BDensity}
Lashi Bandara, \emph{Density problems on vector bundles and manifolds}, Proc.
  Amer. Math. Soc. \textbf{142} (2014), no.~8, 2683--2695. \MR{3209324}

\bibitem[Ban16]{BRough}
\bysame, \emph{Rough metrics on manifolds and quadratic estimates}, Math. Z.
  \textbf{283} (2016), no.~3-4, 1245--1281. \MR{3520003}

\bibitem[Ban17]{BCont}
\bysame, \emph{Continuity of solutions to space-varying pointwise linear
  elliptic equations}, Publ. Mat. \textbf{61} (2017), no.~1, 239--258.

\bibitem[BLM17]{BLM}
Lashi Bandara, Sajjad Lakzian, and Michael Munn, \emph{Geometric singularities
  and a flow tangent to the {R}icci flow}, Ann. Sc. Norm. Super. Pisa Cl. Sci.
  (5) \textbf{17} (2017), no.~2, 763--804. \MR{3700383}

\bibitem[BM16]{BMc}
Lashi Bandara and Alan McIntosh, \emph{The {K}ato {S}quare {R}oot {P}roblem on
  {V}ector {B}undles with {G}eneralised {B}ounded {G}eometry}, J. Geom. Anal.
  \textbf{26} (2016), no.~1, 428--462. \MR{3441522}

\bibitem[CH98]{CH}
Thierry Cazenave and Alain Haraux, \emph{An introduction to semilinear
  evolution equations}, Oxford Lecture Series in Mathematics and its
  Applications, vol.~13, The Clarendon Press, Oxford University Press, New
  York, 1998, Translated from the 1990 French original by Yvan Martel and
  revised by the authors. \MR{1691574}

\bibitem[Cha84]{Chavel}
Isaac Chavel, \emph{Eigenvalues in {R}iemannian geometry}, Pure and Applied
  Mathematics, vol. 115, Academic Press Inc., Orlando, FL, 1984, Including a
  chapter by Burton Randol, With an appendix by Jozef Dodziuk. \MR{768584
  (86g:58140)}

\bibitem[CS85]{CS}
Filippo Chiarenza and Raul Serapioni, \emph{A remark on a {H}arnack inequality
  for degenerate parabolic equations}, Rend. Sem. Mat. Univ. Padova \textbf{73}
  (1985), 179--190. \MR{799906}

\bibitem[Dav89]{Davies}
E.~B. Davies, \emph{Heat kernels and spectral theory}, Cambridge Tracts in
  Mathematics, vol.~92, Cambridge University Press, Cambridge, 1989. \MR{990239
  (90e:35123)}

\bibitem[Kat95]{Kato}
Tosio Kato, \emph{Perturbation theory for linear operators}, Classics in
  Mathematics, Springer-Verlag, Berlin, 1995, Reprint of the 1980 edition.
  \MR{1335452}

\bibitem[Kre15]{Kreuter}
Marcel Kreuter, \emph{Sobolev spaces of vector-valued functions}, Master's
  thesis, Ulm University, 2015.

\bibitem[Lie96]{Lieberman}
Gary~M. Lieberman, \emph{Second order parabolic differential equations}, World
  Scientific Publishing Co., Inc., River Edge, NJ, 1996. \MR{1465184}

\bibitem[Mos64]{MR0159139}
J\"urgen Moser, \emph{A {H}arnack inequality for parabolic differential
  equations}, Comm. Pure Appl. Math. \textbf{17} (1964), 101--134. \MR{0159139}

\bibitem[Mos71]{MR0288405}
\bysame, \emph{On a pointwise estimate for parabolic differential equations},
  Comm. Pure Appl. Math. \textbf{24} (1971), 727--740. \MR{0288405}

\bibitem[MVV05]{MVV}
Amir Manavi, Hendrik Vogt, and J\"urgen Voigt, \emph{Domination of semigroups
  associated with sectorial forms}, J. Operator Theory \textbf{54} (2005),
  no.~1, 9--25. \MR{2168857}

\bibitem[Nor97]{Norris}
James~R. Norris, \emph{Heat kernel asymptotics and the distance function in
  {L}ipschitz {R}iemannian manifolds}, Acta Math. \textbf{179} (1997), no.~1,
  79--103. \MR{1484769 (99d:58167)}

\bibitem[Ouh05]{El-Maati}
El~Maati Ouhabaz, \emph{Analysis of heat equations on domains}, London
  Mathematical Society Monographs Series, vol.~31, Princeton University Press,
  Princeton, NJ, 2005. \MR{2124040}

\bibitem[SC92]{SC}
Laurent Saloff-Coste, \emph{Uniformly elliptic operators on {R}iemannian
  manifolds}, J. Differential Geom. \textbf{36} (1992), no.~2, 417--450.
  \MR{1180389 (93m:58122)}

\bibitem[Stu98]{Sturm}
K.~T. Sturm, \emph{Diffusion processes and heat kernels on metric spaces}, Ann.
  Probab. \textbf{26} (1998), no.~1, 1--55. \MR{1617040}

\bibitem[tERS07]{ERS}
A.~F.~M. ter Elst, Derek~W. Robinson, and Adam Sikora, \emph{Small time
  asymptotics of diffusion processes}, J. Evol. Equ. \textbf{7} (2007), no.~1,
  79--112. \MR{2305727}

\bibitem[Yos95]{Yosida}
K\=osaku Yosida, \emph{Functional analysis}, Classics in Mathematics,
  Springer-Verlag, Berlin, 1995, Reprint of the sixth (1980) edition.
  \MR{1336382}

\end{thebibliography}
\providecommand{\bysame}{\leavevmode\hbox to3em{\hrulefill}\thinspace}
\providecommand{\MR}{\relax\ifhmode\unskip\space\fi MR }
\providecommand{\MRhref}[2]{%
  \href{http://www.ams.org/mathscinet-getitem?mr=#1}{#2}
}
\providecommand{\href}[2]{#2}

\setlength{\parskip}{0pt}

\end{document}